\documentclass[11pt,reqno]{amsart}

\usepackage{amsthm,amsfonts,amsmath,amssymb,bookmark,appendix}
\usepackage{amsmath,amsfonts,amsthm,amssymb,amsxtra}
\usepackage{amsxtra, amssymb, mathrsfs}
\usepackage[usenames,dvipsnames]{color}

\newtheorem{thm}{Theorem}[section]
\newtheorem{prop}[thm]{Proposition}

\newtheorem{lemma}[thm]{Lemma}
\newtheorem{corollary}[thm]{Corollary}
\newtheorem{assumption}[thm]{Assumption}

\theoremstyle{definition}

\theoremstyle{remark}
\newtheorem{remark}[thm]{\bf Remark}

\numberwithin{equation}{section}

\makeatletter
\renewcommand*\env@cases[1][1.6]{%
  \let\@ifnextchar\new@ifnextchar
  \left\lbrace
  \def\arraystretch{#1}%
  \array{@{}l@{\quad}l@{}}%
}
\makeatother

\newcommand{\norm}[1]{\mbox{$\left\| #1 \right\|$}}
\newcommand{\sprod}[2]{\mbox{$\left\langle {\,#1},{\, #2} \right\rangle$}}

\newcommand{\N}{\mathbb{N}}

\newcommand{\R}{\mathbb{R}}

\newcommand{\df}{\, \mathrm{d}}
\newcommand{\menge}[2]{\left\{ \, {#1}  \ \big{|} \ {#2} \, \right\}}

\begin{document}

\title[Hardy inequalities for the Heisenberg Laplacian]{Hardy inequalities for the Heisenberg Laplacian on convex bounded polytopes}


\author{Bartosch Ruszkowski}
\address{Bartosch Ruszkowski, Institute of Analysis, Dynamics
and Modeling, Universit\"at Stuttgart, 
PF 80 11 40, D-70569  Stuttgart, Germany}
\email{Bartosch.Ruszkowski@mathematik.uni-stuttgart.de}

\begin{abstract}
We prove a Hardy-type inequality for the gradient of the Heisenberg Laplacian on open bounded convex polytopes on the first Heisenberg Group. The integral weight of the Hardy  inequality is given by the distance function to the boundary measured with respect to the Carnot-Carath\'{e}odory metric. The constant depends on the number of hyperplanes, given by the boundary of the convex polytope, which are not orthogonal to the hyperplane $x_3=0$.  
\end{abstract}
\maketitle

\section{\bf Introduction}
\noindent Consider the first Heisenberg group given by $\R^3$, equipped with the group law 
\begin{equation} \label{group-law}
(x_1,x_2,x_3)\boxplus (y_1,y_2,y_3) \ := \ \Big{ (x_1+y_1,x_2+y_2, x_3+y_3- \frac{1}{2}\, (x_1 y_2-x_2 y_1)\Big )},
\end{equation}
and the sub-gradient $\nabla_{\mathbb{H}}:=(X_1,X_2)$ given by
\begin{align*}
X_1  \ :=  \ \partial_{x_1}+\frac{x_2}{2}\partial_{x_3}, \quad X_2 \ := \ \partial_{x_2}-\frac{x_1}{2}\partial_{x_3},
\end{align*} 
for $x:=(x_1,x_2,x_3)\in \mathbb{R}^3$. We recall that the vector fields $X_1, X_2, X_3:=[X_2,X_1] = \partial_{x_3} $ form a basis of the Lie algebra of left-invariant vector fields on $\mathbb{H}$ and that the sub-elliptic operator
$$\Delta_{\mathbb{H}}:=-X_1^2-X_2^2 $$
is the Heisenberg Laplacian
, also called Kohn Laplacian. There is a considerable amount of literature concerning the Hardy-type inequality 
\begin{align} \label{introkaplanhardy}
\int_{\mathbb{H}} \frac{|u(x)|^{2}}{ \norm{x}_{\mathbb{H}}^{4}} \, (x_1^2+x_2^2)  \df x \ \leq  \  \int_{\mathbb{H}}  \left|{\nabla_{\mathbb{H}}}\, u(x)\right|^{2} \df x \qquad u\in C_0^{\infty}(\mathbb{H}\setminus \{0\}),
\end{align}
where 
$$
\norm{x}^4_{\mathbb{H}}:= (x_1^{2}+x_2^{2})^{2}+16x_3^{2}\, .
$$
For the proof of \eqref{introkaplanhardy} we refer to \cite{garofalo, aermark,niu}, see also various improvements obtained in \cite{dambrozio, xiao}. The anisotropic norm $\norm{x}_{\mathbb{H}}$, which appears in \eqref{introkaplanhardy}, is referred to in the literature as Kor\'{a}nyi-Folland gauge or Kaplan gauge. For the sake of brevity we will use the latter notation and call it Kaplan gauge. 

\smallskip

In this paper we deal with sharp Hardy inequalities for the Heisenberg-Laplacian on bounded domains. In particular we consider the following problem; given an open bounded domain $\Omega \subset \mathbb{R}^3$, we would like to find the best constant $c$ for which the inequality
\begin{align}\label{ersteungleichung}
\int_{\Omega} \frac{|u(x)|^{2}}{\delta_C(x)^{2}} \df x \ \leq \ c^2 \int_{\Omega}   \left|{\nabla_{\mathbb{H}}}\, u(x)\right|^{2} \df x,
\end{align}
holds for all $u\in C_{0}^{\infty}(\Omega)$, where $\delta_C(x)$ is the Carnot-Carath\'{e}odory distance (C-C distance in the sequel) between $x$ and the boundary of $\Omega$, see section \ref{sec-prelim} for its definition. For more details on the C-C distance we refer to \cite{calin}, \cite{capogna}. With respect to the well-studied inequality \eqref{introkaplanhardy}, it is less known about the validity of  \eqref{ersteungleichung}, especially if one is interested in explicit constants. In \cite{danielli} the authors proved that for every $\Omega$ with a $C^{1,1}$ regular boundary there exists $c>0$ such that \eqref{ersteungleichung} holds true. Later it was shown by Yang, \cite{yang} that if $\Omega$ is a ball with respect to the C-C distance, then \eqref{ersteungleichung} holds with $c=2$.

\smallskip 

The fundamental problem of deriving inequalities of the form \eqref{ersteungleichung} lies in the fact, that we a priori don't know much about domains, which are the most natural ones for a Hardy inequality on $\mathbb{H}$. In comparison to the Euclidean setting it is well-known that if $\Omega$ is {\em convex} then
\begin{align}\label{euklidischeungleichung}
\int_{\Omega} \frac{|u(x)|^{2}}{{\rm dist}(x,\partial\Omega)^{2}}\, \df x \ \leq \ 4 \int_{\Omega}   \left|{\nabla}u(x)\right|^{2} \df x
\end{align}
holds for all $u\in C_0^\infty(\Omega)$, and the constant $4$ is sharp independently of $\Omega$, see 
e.g.~\cite{ancona, kufner, davies,marcus, avkhadiev,avkhadievv,hoffmann}. 

In this paper we prove that for open bounded convex polytopes $\Omega$ we obtain a constant depending on the number of hyperplanes of $\partial \Omega$, which are not orthogonal to the hyperplane $x_3=0$. Under an additional geometrical assumption the constant in \eqref{ersteungleichung} for convex poltytops can be improved, see Theorem \ref{assumpconvexpol1}. It is even possible to show that for any $c>2$ there exists an open bounded convex domain such that \eqref{ersteungleichung} is fulfilled, which is an almost sharp result since we prove that for any bounded domain $\Omega$ it holds
\begin{align}\label{approximatingsequence}
\inf\limits_{u\in  C_{0}^{\infty}(\Omega)} \frac{\int_{\Omega} |\nabla_{\mathbb{H}}\, u(x)|^{2} \df x}{\int_{\Omega}  {|u(x)|^{2}}{{\delta_C(x)^{-2} \df x}}} \leq \frac 14.
\end{align}
This shows that at least some convex domains are more compatible with the Heisenberg group structure than we expect them to be.

In \cite{luan} Luan and Yang proved on the half-space $\Omega:=\{x\in\mathbb{H} | x_3>0\}$ that for any $u\in C_0^\infty(\Omega)$ holds 
\begin{align}\label{geometriclarson}
 \int_{\Omega}  \frac{x_1^2+x_2^2}{4x_3^2}|u(x)|^2 \df x \leq  4\int_{\Omega} |\nabla_{\mathbb{H}}u(x)|^{2} \df x.
\end{align}
This result was recently generalized by Larson \cite{larson} to any bounded convex domain. Under an additional convexity condition, where $H(x)$ denotes the tangent plane to $x$, we can replace the weight on the left-hand side by
\begin{align}\label{reduceddistance}
 \omega(x) &:=  \inf\limits_{y\in \partial\Omega\cap H(x)}  d_{C}(x,y),
\end{align}
see Theorem \ref{mainresultinthree}. This result turns out to be \eqref{geometriclarson} for the case of the half-space.
 
The paper is organized as follows. In the next section we introduce necessary notation. Main results are formulated in section \ref{sec-main} and the proof of each Theorem is done in a separate section.

\section{\bf Preliminaries and notation}
\label{sec-prelim}

\noindent  The \textit{tangent plane} to $x:=(x_1,x_2,x_3)\in\mathbb{H}$ is given by
\begin{equation} \label{tangentplane}
\begin{split}
H(x):=& \menge{y\in \mathbb{H}}{\sprod{
 \left(-\frac{x_2}{2}, \frac{x_1}{2},1\right)
}{y-x}=0},\\
=& \menge{y\in \mathbb{H}}{x_1y_2-x_2y_1=2(x_3-y_3)},
\end{split}
\end{equation}
where $\sprod{\cdot}{\cdot}$ is the Euclidean scalar product in $\R^3$.
\smallskip

\noindent Let us briefly recall the definition of the C-C distance $d_C(x,y)$. We call  a Lipschitz curve $\gamma : [a,b] \rightarrow \mathbb{H}$ parametrized by 
$\gamma(t)= (\gamma_1(t),\gamma_2(t),\gamma_3(t))$  {\it horizontal}  if 
$$
\gamma^{\prime}(t)\in \mathrm{span}\left\{ \left(1,0, \frac{\gamma_2(t)}{2}\right), \, \left(0,1, -\frac{\gamma_1(t)}{2}\right)\right\}\, .
$$ 
The C-C distance between $x$ and $y$ is then defined as
\begin{align} \label{def-cc}
 d_C(x,y) := \inf\limits_{\gamma} \int_{a}^b \sqrt{\gamma^\prime_1(t)^2+\gamma^\prime_2(t)^2} \ \df t,
\end{align}
where the infimum is taken over all horizontal curves $\gamma$ connecting $x$ and $y$. 
\smallskip

\noindent  We define the C-C and Kaplan distance functions for an open bounded domain $\Omega$ by 
\begin{equation} \label{dists}
\delta_C(x) :=  \inf\limits_{y\in \partial\Omega} d_C(x,y) , \qquad \delta_K(x):= \inf\limits_{y\in \partial\Omega} \norm{(-y)\boxplus x}_{\mathbb{H}}\, .
\end{equation}
For $x\in \Omega^c$ we extend these functions by $\delta_C(x):=0$ and $\delta_K(x):=0$. With these prerequisites we can state the main results of our paper.

\section{\bf Main results}
\label{sec-main}

\begin{thm}\label{mainresultinthree}
Let $\Omega \subset \mathbb{H}$ be open bounded and let the connected components of $H(x) \cap \Omega$ be convex for all $x\in \Omega$. Then holds
\begin{align}\label{hardy-omega}
\int_{\Omega} \frac{|u(x)|^{2}}{{\omega(x)^{2}}}   \df x \ \leq 4\int_{\Omega} |\nabla_{\mathbb{H}}\, u(x)|^{2} \df x
\end{align}
for all $u\in C_{0}^{\infty}(\Omega)$, where $\omega(\cdot)$ is defined in \eqref{reduceddistance} and it holds

\begin{equation}
\begin{split}\label{introidenteuclidC-C}
\omega(x)= \inf\limits_{y\in \partial\Omega\cap H(x)} \norm{(-y)\boxplus x}_{\mathbb{H}}= \inf\limits_{y\in \partial\Omega\cap H(x)}   \sqrt{(x_1-y_1)^2+(x_2-y_2)^2}.
\end{split}
\end{equation}
\end{thm}
\noindent  We call the weight $\omega(\cdot)$ the reduced C-C distance.  The proof of \eqref{hardy-omega} is done in the following way. We proof the  Hardy inequality for each separate $X_j$, where the distance function is given by the C-C metric generated by $X_j$ for $j\in\{1,2\}$. Then we apply the hyperplane separation theorem in the same way as Davies did for the proof of \eqref{euklidischeungleichung} for convex domains, see \cite{davies}.

\begin{thm} \label{maineins}
Let $ \Omega \subset \mathbb{H}$ be an open bounded convex polytope and let $m\in\mathbb{N}$ be the number of hyperplanes of $\partial \Omega$, which are not orthogonal to the hyperplane $x_3=0$. Then holds
\begin{align}\label{hardyconvexpoly}
\frac{1}{5}\left( \frac{48\sqrt{6}}{c_m}+  1  \right)^{-4/3}\int_{\Omega} \frac{|u(x)|^2}{\delta_C(x)^2} \df x \ \leq  \int_{\Omega} |\nabla_{\mathbb{H}}\, u(x)|^{2} \df x,
\end{align}
for all $u\in C_{0}^{\infty}(\Omega)$, where $c_m$ is a positive solution of
\begin{align*}
  c_m^{4/3}  \pi m\sqrt{c_m^2+16} \left( 1+   \frac{c_m}{48\sqrt{6}}  \right)^{2/3}= 16/3.
\end{align*}
\end{thm}

\noindent In addition we proof that for $c_m$ holds the following 
 \begin{align}\label{thecmconstant}
 \frac1c_m  \leq   3\pi m^{8/9} 2^{-11/6} \left( 1+ \frac{1}{12\sqrt{6}}  \right)^{2/3},
\end{align}
which yields a result with an explicit constant in \eqref{hardyconvexpoly}.

The strategy of the proof of Theorem \ref{maineins} consists of two steps. We use Theorem \ref{mainresultinthree} for a bounded convex polytope. Then we take into account the following Hardy inequality
 \begin{align}\label{onepointyang}
\int_{\Omega} \frac{|u(x)|^{2}}{d_C(x,0)^{2}} \df x \ \leq \ \int_{\Omega}   \left|{\nabla_{\mathbb{H}}}\, u(x)\right|^{2} \df x,
\end{align}
for all $u\in C_{0}^{\infty}(\Omega)$, which was proved in \cite{goldstein,ruzhansky,yang}. The sum of the weight functions is then a comparable to distance function to the hyperplanes of the given polytope respectively the Kaplan gauge, which is equivalent to the distance function respectively  the C-C metric. 
 
We can improve the constant in Theorem \ref{maineins} under an additional geometrical assumption, which is discussed in section \ref{Copoimproved}. The main consequence of that result is the following;
 \begin{thm}\label{1approxdom}
For any $\varepsilon>0$ there exists an open bounded convex domain $\Omega$ such that
\begin{align*}
\int_{\Omega} \frac{|u(x)|^2}{\delta_C(x)^2} \df x \ \leq (2+\varepsilon)^2\int_{\Omega} |\nabla_{\mathbb{H}}\, u(x)|^{2} \df x
\end{align*}
holds for all $u\in C_0^\infty(\Omega)$.
\end{thm}
\noindent The last result has an almost optimal constant, since we proof that for any open bounded domain $\Omega$ holds 
\begin{align*}
\inf\limits_{u\in  C_{0}^{\infty}(\Omega)} \frac{\int_{\Omega} |\nabla_{\mathbb{H}}\, u(x)|^{2} \df x}{\int_{\Omega}  {|u(x)|^{2}}{{\delta_C(x)^{-2} \df x}}} \leq \frac 14,
\end{align*}
see Theorem \ref{lem-sharp} .
\section{\bf Restricted C-C distance and its connection to the Euclidean distance}


\subsection{The natural restriction of \texorpdfstring{$\partial\Omega$}{}} In this section we show that  the reduced distance $\omega(\cdot)$, defined by \eqref{reduceddistance} can be expressed in terms of a simple explicit formula. In particular, we show that $\omega$  coincides with the distance to the boundary in the Kaplan gauge as well as the projection onto the $x_1$-$x_2$-hyperplane of the Euclidean metric.

\begin{thm}\label{lemmaeuclidC-C}
Let $\Omega \subset \mathbb{H}$ be open bounded, then holds
\begin{align*}
\omega(x)\ = \ \inf_{y\in \partial\Omega\cap H(x)} \sqrt{(x_1-y_1)^2+(x_2-y_2)^2} \ = \  \inf_{y\in \partial\Omega\cap H(x)}\norm{(-y)\boxplus x}_{\mathbb{H}},
\end{align*}
for all $x\in \Omega$.
\end{thm}

\noindent For the proof we need the following;

\begin{lemma}\label{lemmasharpestimate}
For all $x,y \in \mathbb{H}$ it holds
\begin{align}\label{Ungleichung}
\frac{1}{\pi^{2}}\, d_C(x,y)^{4} \leq  \norm{(-y)\boxplus x}_{\mathbb{H}}^{4} \leq d_C(x,y)^{4}. 
\end{align}
Moreover, both inequalities are sharp.
\end{lemma}

\begin{proof}{ Using the left-invariance of $d_C(x,y)$ respectively the group law on $\mathbb{H}$ we transform \eqref{Ungleichung} into
\begin{align} 
 \frac{1}{\pi^{2}}\, d_C(y^{-1}\boxplus x,0)^{4} \leq  \norm{(-y)\boxplus x}_{\mathbb{H}}^{4} \leq d_C(y^{-1}\boxplus x,0)^{4}.
\end{align}
We know that $y^{-1}=-y$. Therefore it is sufficient to prove 
\begin{align*}
\frac{1}{\pi^{2}}\, d_C(z,0)^{4} \leq  \norm{z\boxplus 0}_{\mathbb{H}}^{4} \leq d_C(z,0)^{4} \qquad \forall\ z\in \mathbb{H}.
\end{align*}
The arc joining geodesics starting from the origin 
were computed in  \cite{montii} and \cite{marenich}. The parametrization of these arcs is given by
\begin{align}\label{geod}
z=\gamma_{k,\theta}(t):= 
\begin{cases}
z_1(t,k,\theta) =\displaystyle \frac{\cos(\theta)-\cos(kt+\theta)}{k},\\[0,3cm]
z_2(t,k,\theta)=\displaystyle \frac{\sin(kt+\theta)-\sin(\theta)}{k},\\[0,3cm]
z_3(t,k,\theta)=\frac{kt-\sin(kt)}{2k^{2}}\, ,
\end{cases}
\end{align}
where $t\in[0, \frac{2\pi}{{|k|}}]$ , $\theta \in [0,2\pi)$ and $k \in \mathbb{R}\setminus \{0\}$. This means that for the given point $z:=\gamma_{k,\theta}(t)\in \mathbb{H}$  holds $d(\gamma_{k,\theta}(t),0)=t$.  We extend this formula to the case $k=0$ by taking the limit for $k\to 0$. This gives
\begin{align} \label{geod2}
z=\gamma_{0,\theta}(t):= 
\begin{cases}
z_1(t,0,\theta) =\displaystyle t\sin(\theta),\\
z_2(t,0,\theta)=\displaystyle  t\cos(\theta),\\
\displaystyle z_3(t,0,\theta)=0.
\end{cases}
\end{align}
For the computation of  $d_C(z,0)$ we use \eqref{geod}. It is then sufficient to calculate the supremum and the infimum of
\begin{align*}
\frac{ \norm{\gamma_{k,\theta}(t)\boxplus 0}_{\mathbb{H}}^{4}}{ d_C(z,0)^{4}}= \frac{  4\left( 1-\cos(kt) \right)^{2}+4\left(kt-\sin(kt)  \right)^{2}  }{(tk)^{4}}.
\end{align*}
This leads to estimating the function
\begin{align*}
g(\tau):=\frac{4}{\tau^{4}}\left( \left(1 -\cos(\tau) \right)^{2}+\left(\tau-\sin(\tau)  \right)^{2}   \right),
\end{align*}
with $0\leq \tau \leq 2 \pi$, because $t\in[0, \frac{2\pi}{|k|}]$. To proceed we show that the function $g(\tau)$
is non-increasing on $[0, 2\pi]$. By differentiating the function $g(\tau)$ several times we find that the latter is non-increasing on $[0,2\pi]$ which implies that the same is true for $g$. Hence
\begin{align}
\frac{1}{\pi^{2}} = {g(2\pi) }  \leq g(\tau) \leq \lim_{\tau \to 0+} g(\tau)= 1.
\end{align}
The sharpness of that inequality is an immediate consequence.
}\end{proof}

\begin{proof}[\bf Proof of Theorem {\normalfont\ref{lemmaeuclidC-C}}]
Let $x\in \Omega$ and let $y\in \partial\Omega \cap H(x)$. Consider the curve $\gamma:[0,1] \to \mathbb{H}$ given by the parametrization $\gamma(t)=(1-t)x+ty,\  t\in[0,1]$. Obviously $\gamma$ connects $x$ and $y$. Moreover, since $y\in H(x)$ it is easily verified that $\gamma$ is horizontal. Indeed, we have 
$$
\gamma'(t) = (y_1-x_1)  \left(1,0, \frac{\gamma_2(t)}{2}\right)+(y_2-x_2)  \left(0,1, -\frac{\gamma_1(t)}{2}\right)\, .
$$
By definition of the C-C distance, see equation \eqref{def-cc}, it thus follows that
\begin{align}
d_C(x,y) \ \leq \ \sqrt{(x_1-y_1)^2+(x_2-y_2)^2}.
\end{align}
Using $y\in \partial\Omega \cap H(x)$ we see that 
\begin{align}
\sqrt{(x_1-y_1)^2+(x_2-y_2)^2} \ = \ \norm{(-y)\boxplus x}_{\mathbb{H}}.
\end{align}
Then we apply Lemma \ref{lemmasharpestimate} to obtain the following chain of inequalities
\begin{align}
d_C(x,y) \ \leq \  \sqrt{(x_1-y_1)^2+(x_2-y_2)^2} \ = \  \norm{(-y)\boxplus x}_{\mathbb{H}} \leq \ d_C(x,y).
\end{align}
Taking the infimum over $y\in \partial\Omega \cap H(x)$ yields the result.
\end{proof}

\subsection{The Hardy inequality involving \texorpdfstring{$\omega$}{}} 

\noindent We need the following auxiliary result.

\begin{lemma}\label{pseudodist}
Let $\Omega$ be an open bounded domain in $\mathbb{H}$. Then holds
\begin{align}
\int_{\Omega}\left( \frac{|u(x)|^{2}}{d_{1}(x)^{2}}+ \frac{|u(x)|^{2}}{d_{2}(x)^{2}}\right) \df x \ \leq 4\int_{\Omega} |\nabla_{\mathbb{H}}\, u(x)|^{2} \df x
\end{align}
for all $u\in C_{0}^{\infty}(\Omega)$, where the distances $d_{1}(x)$ and $d_{2}(x)$ are given by
\begin{align}
d_{1}(x) \ &:= \ \inf_{s\in\R} \, \{|s|>0 \, |\, x+s(1,0,x_2/2) \notin \Omega  \}, \\
d_{2}(x) \ &:= \ \inf_{s\in\R}\, \{|s|>0 \,|\, x+s(0,1,-x_1/2) \notin \Omega \}.
\end{align}
\end{lemma}
\begin{proof} Let $u\in C_{0}^\infty(\Omega)$. First we show that 
\begin{align}\label{gleichungeins}
\int_{\Omega} \frac{|u(x)|^{2}}{d_{1}(x)^{2}} \df x \ \leq 4\int_{\Omega} |X_1u(x)|^{2} \df x.
\end{align}
To this end we define the following coordinate transformation
\begin{align}\label{neuekoordinaten}
F(t,\varphi,\theta) \ := \
\begin{cases}
  x_{1}(t,\varphi,\theta) = t+\varphi, \\
 x_2(t,\varphi,\theta)= \theta, \\
 x_3(t,\varphi,\theta) = {t\theta/2},
\end{cases}
\end{align}
where $(t,\varphi,\theta)\in A:=\{(t,\varphi,\theta)\in \mathbb{R}^{3} \ | \ \theta\neq 0\}$. It can be easily checked that $F: A \mapsto \mathrm{Ran}(A)$ is a diffeomorphism and that the determinant of $F$ is equal to $\theta/2$. For a given $x\in \Omega^{c}$ we set $u(x)=0$. If $x=F(t,\varphi,\theta)$ for fixed $\theta \in \mathbb{R}\setminus \{0\}$ and $\varphi \in \mathbb{R}$ we see that there exists a constant $c\in \mathbb{R}$ such that $F(c,\varphi,\theta)=\hat{x}\in\partial\Omega$ fulfills $d_1(x)=d_C(x,\hat{x})$. By  $\{a_j\}_{j\in\mathbb{N}}$ we denote the increasing sequence such that $F(a_j,\varphi,\theta)\in \partial \Omega$.
Thus for a fixed $x\in \Omega$ we immediately see, that there exists  a $k\in \mathbb{N}$ such that
\begin{align*}
d_{1}(F(t,\varphi,\theta))=d_C(F(t,\varphi,\theta),F(a_k,\varphi,\theta))&=d_C(F(t,\varphi,\theta),F(t,\varphi,\theta)+(a_k-t)(1,0,\theta/2))\\
&= |a_k-t|
\end{align*}
By the last observation we apply then the transformation $F$ to find out that to prove \eqref{gleichungeins} it suffices to show that
\begin{equation}\label{hardnachkootrafo}
\begin{split}
 \int_{\mathbb{R}}\int_{\mathbb{R}}  &\sum_{j=1}^{\infty} \int_{a_j}^{a_{j+1}} \frac{|u(t,\varphi,\theta)|^{2}}{\delta_j(t)^{2}} \df  t \frac{|\theta|}{2} \df \theta \df \varphi  \\
  &\leq \ 4  \int_{\mathbb{R}}\int_{\mathbb{R}}  \sum_{j=1}^{\infty} \int_{a_j}^{a_{j+1}} {|\partial_tu(t,\varphi,\theta)|^{2}}\df  t \frac{|\theta|}{2}\df  \theta \df \varphi ,
\end{split}
\end{equation}
where $\delta_{j}(t):=\inf(a_{j+1}-t,t-a_j)$. Hence the one-dimensional Hardy inequality in the $t$-direction then implies that \eqref{hardnachkootrafo} holds true which in turn yields  \eqref{gleichungeins}.  It remains to prove
\begin{align}\label{gleichungzwei}
\int_{\Omega} \frac{|u(x)|^{2}}{d_{2}(x)^{2}} \df x \ \leq 4\int_{\Omega} |X_2u(x)|^{2} \df x.
\end{align}
This is done in the same way as \eqref{gleichungeins} replacing the transformation of \eqref{neuekoordinaten} by 
\begin{align}
\tilde{F}(t,\varphi,\theta) \ := \
\begin{cases}
  x_{1}(t,\varphi,\theta) = \theta, \\
 x_2(t,\varphi,\theta)= t+\varphi , \\
 x_3(t,\varphi,\theta) = {-t\theta/2},
\end{cases}
\end{align}
for $(t,\varphi,\theta)\in A$. Summing up \eqref{gleichungeins} and \eqref{gleichungzwei} then completes the proof. 
\end{proof}

 \begin{proof}[\bf Proof of Theorem {\normalfont\ref{mainresultinthree}}]
Let $a:=(a_1,a_2,a_3)\in \partial \Omega \cap H(x)$ be such that
$$
 \omega(x)=\sqrt{(x_1-a_1)^2+(x_2-a_2)^2}.
$$ 
The existence of such $a$ is guaranteed by the compactness of $\overline{\Omega}$ and the continuity of the distance. We know that all connected components of $H(x)\cap  \Omega$ are convex. Therefore we assume without loss of generality that $H(x)\cap \Omega$ consists of a single connected component, which is convex. Next we apply the hyperplane separation theorem, which implies that the hyperplane
\begin{align}\label{planeone}
T:=\menge{y\in \mathbb{H}}{\sprod{\begin{pmatrix}
 x_1-a_1\\
 x_2-a_2\\
 0
\end{pmatrix}}{\begin{pmatrix}
 y_1-a_1\\
 y_2-a_2\\
 y_3
\end{pmatrix}}=0}
\end{align}
separates $H(x)\cap \Omega$ from the point $a \in \partial \Omega$. We consider the definition of $d_1(x)$, see Lemma \ref{pseudodist}, and compute the intersection point of the line $c(s)=x+s(1,0,1/2x_2)^t$ for $s\in \mathbb{R}$ with the hyperplane \eqref{planeone}. This yields
\begin{align}
s= -\frac{(x_1-a_1)^2+(x_2-a_2)^2}{x_1-a_1}.
\end{align}
At this point we apply the hyperplane separation theorem again to infer that 
$$
d_{1}(x)\leq   \frac{(x_1-a_1)^2+(x_2-a_2)^2}{|x_1-a_1|}.
$$
Now we do the same computation for $d_2(x)$ and obtain
\begin{align}
d_{2}(x) \leq  \frac{(x_1-a_1)^2+(x_2-a_2)^2}{|x_2-a_2|}.
\end{align}
Altogether we get
\begin{equation*}
 \frac{1}{d_{1}(x)^{2}}+  \frac{1}{d_{2}(x)^{2}}  \ \geq \ \frac{(x_2-a_2)^2}{((x_1-a_1)^2+(x_2-a_2)^2)^2}+\frac{(x_1-a_1)^2}{((x_1-a_1)^2+(x_2-a_2)^2)^2}=\frac{1}{\omega(x)^2}.
\end{equation*}
We recall that the point $a\in \partial \Omega \cap H(x)$ was chosen such that  $\omega(x)=\sqrt{(x_1-a_1)^2+(x_2-a_2)^2}$, which proves inequality \eqref{hardy-omega}. 
\end{proof}
\begin{remark}For $p\geq2$ it is possible to get an $L^{p}$ version of Theorem \ref{mainresultinthree} as well. In Lemma \ref{pseudodist} we use the $L^p$ version of the one-dimensional Hardy inequality, which holds for $p>1$. Then we mimic the last proof and apply for $p\geq 2$ Jensen's inequality
\begin{align*}
(a^2+b^2)^{p/2}= 2^{p/2}(a^2/2+b^2/2)^{p/2}\leq 2^{p/2-1}(a^p+b^p),
\end{align*}
for $a,b>0$.
\end{remark}

\section{ \bf Proof of the Hardy inequalities for open bounded convex polytopes}\noindent
In this section we give the proof of Theorem \ref{maineins}. First we have to give some  lower estimates for the Kaplan distance function to hyperplanes, which are not orthogonal to the $x_3=0$ hyperplane. Therefore we need the following;
\begin{lemma}
Let $p>0$ and $q\in\mathbb{R}\setminus\{0\}$. Consider 
\begin{align*}
z^{3}+pz=q,
\end{align*}
for $z\in\mathbb{R}$. Then there exists a unique real solution for which holds
\begin{align*}
|z|\geq    \frac{|q^{1/3}|}{3}\left( 1+\frac{p\sqrt{p}}{|q|3\sqrt{3}}\right)^{-2/3}.
\end{align*}
\end{lemma}
\begin{proof}
First we consider the case $q>0$. Then Cardano's formula gives the unique real solution
\begin{align*}
z&= \left(q/2+\sqrt{q^2/4+p^3/27} \right)^{1/3}+\left(q/2-\sqrt{q^2/4+p^3/27} \right)^{1/3}\\
&= \frac{1}{3}\int_{-q/2+\sqrt{q^2/4+p^3/27}}^{q/2+\sqrt{q^2/4+p^3/27}}  s^{-2/3} \df s\geq \frac{q}{3}\left( q/2+\sqrt{q^2/4+p^3/27}\right)^{-2/3}\\
&\geq \frac{q}{3}\left( q+\sqrt{p^3/27}\right)^{-2/3}
\end{align*}
The case $q<0$ will be treated in the same way.

\end{proof}

\begin{prop}\label{propohyperplane}
Let $x\in \mathbb{H}$ and $a>0$. We consider
\begin{align*}
\Pi:=\{y\in \mathbb{H} | \ n_1y_1+n_2y_2+n_3y_3=c\},
\end{align*}
where $n_1,n_2,n_3,c\in \mathbb{R}$ with $n_3\neq0$. For the case $(-2n_2/n_3+x_1)^2+(2n_1/n_3+x_2)^2\leq a|-c/n_3+x_3+x_1n_1/n_3+x_2n_2/n_3|$ it holds
\begin{align*}
\left( \inf\limits_{y\in {\Pi}} \norm{(-y)\boxplus x}_{\mathbb{H}}\right)^2\geq    \frac{|-c/n_3+x_3+x_1n_1/n_3+x_2n_2/n_3|}{4\cdot 3^3}\left( 1+   \frac{a}{48\sqrt{6}}  \right)^{-2}
\end{align*}
and for $(-2n_2/n_3+x_1)^2+(2n_1/n_3+x_2)^2\geq a|-c/n_3+x_3+x_1n_1/n_3+x_2n_2/n_3|$ it holds
\begin{align*}
\left( \inf\limits_{y\in \Pi} \norm{(-y)\boxplus x}_{\mathbb{H}}\right)^2 \geq  \frac{4|-c/n_3+x_3+x_1n_1/n_3+x_2n_2/n_3|^{2}}{(-2n_2/n_3+x_1)^2+(2n_1/n_3+x_2)^2}     \left( \frac{48\sqrt{6}}{a}+  1  \right)^{-4/3},
\end{align*}
\end{prop}
\begin{proof}First of all we consider the case $n_1=n_2=c=0$ and $n_3=1$. Let $y\in\mathbb{H}$ such that $y_3=0$ and fix $x:=(x_1,x_2,x_3)\in \mathbb{H}$ with $x_3\neq 0$.  We set $z_1:=y_1-x_1$ and $z_2:=y_2-x_2$ and consider

\begin{equation}\label{poujwwodmn}
\begin{split}
\norm{(-y)\boxplus x}_{\mathbb{H}}^4&=( {(y_1-x_1)^2+(y_2-x_2)^2})^2+\frac{1}{16}({y_3-x_3-1/2y_1x_2+1/2y_2x_1} )^2\\
&= ( z_1^2+z_2^2)^2+\frac{1}{16}(-x_3-1/2z_1x_2+1/2z_2x_1 )^2.
\end{split}
\end{equation}
Then we compute the minimum of the right-hand side in dependence of $x$. We assume that $x_1\neq 0$, since $x_1=0$ is a null set and $\delta_K$ is continuous, see \eqref{heisenbergeuclideanballs}, \eqref{asdasdasdasfdfd}, Lemma \ref{lemmasharpestimate}. The derivatives respectively $y_1$ and $y_2$ yield then
\begin{equation}\label{poujwwodmn2}
 \begin{split}
 &( z_1^2+z_2^2)4z_1-x_2\frac{1}{16}(-x_3-1/2z_1x_2+1/2z_2x_1 )=0,\\
&( z_1^2+z_2^2)4z_2+x_1\frac{1}{16}(-x_3-1/2z_1x_2+1/2z_2x_1)=0.
\end{split}
\end{equation}
Since $x_1\neq 0$ we easily deduce that $z_1^2+z_2^2\neq0$ and obtain
\begin{align*}
z_1=\frac{-z_2x_2}{x_1}.
\end{align*}
Inserting this in \eqref{poujwwodmn} yields
\begin{align}\label{poujwwodmn3}
 \norm{(-y)\boxplus x}_{\mathbb{H}}^4 = z_2^4\frac{(x_2^2+x_1^2)^2}{x_1^4}+\frac{1}{16}\left(-x_3+1/2z_2\frac{x_2^2+x_1^2}{x_1}\right)^2.
\end{align}
We compute the critical points respectively $y_2$ and obtain
\begin{align}\label{poujwwodmn4}
 \norm{(-y)\boxplus x}_{\mathbb{H}}^4 &= z_2^4\frac{(x_2^2+x_1^2)^2}{x_1^4}+256z_2^{6}\frac{(x_2^2+x_1^2)^2}{x_1^6},
\end{align}
where $z_2$ is the unique real solution of
\begin{align*}
z_2^3+z_2\frac{x_1^2}{128}=\frac{x_3x_1^3}{x_2^2+x_1^2}\frac{1}{64}, \quad p:= \frac{x_1^2}{128}, \quad q:=\frac{x_3x_1^3}{x_2^2+x_1^2}\frac{1}{64} .
\end{align*}
Using the estimate in the previous Lemma we obtain 
\begin{align}\label{poujwwodmn5}
|z_2|\geq    \frac{|x_3|^{1/3}|x_1|}{12\sqrt[1/3]{x_1^2+x_2^2}}\left( 1+   \frac{x_1^2+x_2^2}{|x_3|48\sqrt{6}}  \right)^{-2/3}
\end{align}
For the case $x_1^2+x_2^2\leq a|x_3|$ we use   
\begin{align*}
  \norm{(-y)\boxplus x}_{\mathbb{H}}^4 &\geq 256z_2^{6}\frac{(x_2^2+x_1^2)^2}{x_1^6},
\end{align*}
and \eqref{poujwwodmn5} to get
\begin{align*}
\left( \inf\limits_{y\in \mathbb{H}, y_3=0} \norm{(-y)\boxplus x}_{\mathbb{H}}\right)^2\geq    \frac{|x_3|}{4\cdot 3^3}\left( 1+   \frac{a}{48\sqrt{6}}  \right)^{-2}.
\end{align*}
For the case $x_1^2+x_2^2\geq a|x_3|$ we use \eqref{poujwwodmn5}  again for
\begin{align*}
  \norm{(-y)\boxplus x}_{\mathbb{H}}^4 &\geq z_2^4\frac{(x_2^2+x_1^2)^2}{x_1^4},
\end{align*}
which yields
\begin{align*}
\left( \inf\limits_{y\in \mathbb{H}, y_3=0} \norm{(-y)\boxplus x}_{\mathbb{H}}\right)^2 \geq  \frac{4|x_3|^{2}}{(x_2^2+x_1^2)}     \left( \frac{48\sqrt{6}}{a}+  1  \right)^{-4/3},
\end{align*}
To obtain the result for a general hyperplane we consider
\begin{align*}
\inf\limits_{y\in \Pi}\norm{(-y)\boxplus x}_{\mathbb{H}} = \inf\limits_{y\in \Pi}\norm{(-(v\boxplus y))\boxplus(v\boxplus x)}_{\mathbb{H}}= \inf\limits_{(-v)\boxplus q\in \Pi}\norm{(-q)\boxplus(v\boxplus x)}_{\mathbb{H}},
\end{align*}
where $q:=(q_1,q_2,q_3)\in\mathbb{H}$, and $v\in\mathbb{H}$ is set
\begin{align*}
v:= \frac{1}{n_3}\left(-{2n_2}, {2n_1}, -c \right).
\end{align*}
Then $(-v)\boxplus q\in \Pi$ is equivalent to $q_3=0$, which yields the result.
\end{proof}


\begin{proof}[Proof of Theorem {\normalfont\ref{maineins}:}] Let us assume that $\Omega$ is an open bounded convex polytope. Let $m\in\mathbb{N}$ be the number of hyperplanes of $\partial \Omega$, which are not orthogonal to the hyperplane $y_3=0$. We denote these hyperplanes by $\Pi_{j}$ for $1\leq j \leq m$. Thus there exist  $n_{1,j},n_{2,j},n_{3,j},c_j\in\mathbb{R}$ such that
\begin{align*}
\Pi_j:=\{y\in \mathbb{H} | \ n_{1,j}y_1+n_{2,j}y_2+n_{3,j}y_3=c_j\},
\end{align*}
where $n_{3,j}\neq 0$ for $1\leq j \leq m$. By $n_j\in\R^3$ we denote the unit normal of $\Pi_j$. We use Lemma \ref{pseudodist} and inequality \eqref{onepointyang} to obtain
\begin{align}\label{weight12}
\int_{\Omega}\left( \frac{1}{d_{1}(x)^{2}}+ \frac{1}{d_{2}(x)^{2}} + \frac{1}{m}\sum_{j=1}^m\frac{1}{d_C(x,a_j)^2} \right)|u(x)|^2 \df x \ \leq 5\int_{\Omega} |\nabla_{\mathbb{H}}\, u(x)|^{2} \df x
\end{align}
for $u\in C_0^\infty(\Omega)$, where 
\begin{align*}
a_j:=  \frac{1}{n_{3,j}}\left({2n_{2,j}}, -{2n_{1,j}}, c_j \right).
\end{align*}
The aim is to give a pointwise estimate for the weights on the left-hand side from below. We take $b\in \partial \Omega$ such that $\delta_{K}(x)=\norm{(-b)\boxplus x}_{\mathbb{H}}$, which exists, since $\partial\Omega$ is compact and $\delta_K$ is continuous. 

The first case is $b\in \Pi_j$ for a fixed $j$. Since $\Omega$ is convex we compute the intersection point of $d_1(x)$ with $\Pi_j$ as well as the intersection point of $d_2(x)$ with $\Pi_j$. The hyperplane separation theorem yields then
\begin{align*}
\frac{1}{d_{1}(x)^{2}}+ \frac{1}{d_{2}(x)^{2}} \geq \frac{1}{4}\frac{(-2n_{2,j}+x_1n_{3,j})^2+(2n_{1,j}+x_2n_{3,j})^2}{|-c_j+\sprod{x}{n_j}|^{2}}.
\end{align*}
Let $a>0$. We use Proposition \ref{propohyperplane} for the case $(-2n_{2,j}/n_{3,j}+x_1)^2+(2n_{1,j}/n_{3,j}+x_2)^2\geq a|-c/n_3+x_3+x_1n_1/n_3+n_2x_2/n_3|$ which yields
\begin{align*}
\frac{1}{d_{1}(x)^{2}}+ \frac{1}{d_{2}(x)^{2}} \geq \left( \frac{48\sqrt{6}}{a}+  1  \right)^{-4/3}\left( \inf\limits_{y\in \Pi_j} \norm{(-y)\boxplus x}_{\mathbb{H}}\right)^{-2}.
\end{align*}
 For the case $(-2n_{2,j}/n_{3,j}+x_1)^2+(2n_{1,j}/n_{3,j}+x_2)^2\leq a|-c/n_3+x_3+x_1n_1/n_3+n_2x_2/n_3|$ we use Lemma \ref{lemmasharpestimate}
 \begin{align*}
 \frac{1}{m}\sum_{k=1}^m\frac{1}{d_C(x,a_k)^2} &\geq \frac{1}{\pi m} \frac{1}{\norm{-a_j\boxplus x}^2_{\mathbb{H}}} 
 \geq \frac{1}{\pi m\sqrt{a^2+16}}|-c/n_3+x_3+x_1n_1/n_3+n_2x_2/n_3|^{-1}
\end{align*}
and then again Proposition \ref{propohyperplane} yields
 \begin{align*}
 \frac{1}{m}\sum_{k=1}^m\frac{1}{d_C(x,a_k)^2} \geq  \frac{1}{4\cdot 3^3\pi m\sqrt{a^2+16}} \left( 1+   \frac{a}{48\sqrt{6}}  \right)^{-2}\left( \inf\limits_{y\in {\Pi}_j} \norm{(-y)\boxplus x}_{\mathbb{H}}\right)^{-2}.
\end{align*}
We choose  $a>0$ such that
\begin{align*}
  \frac{1}{4\cdot 3^3\pi m\sqrt{a^2+16}} \left( 1+   \frac{a}{48\sqrt{6}}  \right)^{-2}=\left( \frac{48\sqrt{6}}{a}+  1  \right)^{-4/3}
\end{align*}
is fulfilled, which obviously exists. The positive constant, which fulfills that equation is denoted by $c_{m}$. If we summaries our estimates the weight function in \eqref{weight12} is then bounded from below by
\begin{align*}
 \left( \frac{48\sqrt{6}}{c_m}+  1  \right)^{-4/3}\left( \inf\limits_{y\in {\Pi}_j} \norm{(-y)\boxplus x}_{\mathbb{H}}\right)^{-2}&\geq  \left( \frac{48\sqrt{6}}{c_m}+  1  \right)^{-4/3}\norm{(-b)\boxplus x}_{\mathbb{H}}^{-2}
\end{align*}
where we used $b\in \Pi_j$. We recall that $b$ was chosen such that  $\delta_{K}(x)=\norm{(-b)\boxplus x}_{\mathbb{H}}$.

The second case is $b:=(b_1,b_2,b_3)\in\partial\Omega$ when the hyperplane, which contains $b$, is orthogonal to the hyperplane $x_3=0$. We denote that hyperplane by $\Pi$. Because of the orthogonality condition, the hyperplane is parametrized by
\begin{align*}
 \Pi_j:=\{y\in \mathbb{H} | \ (b_1-x_1)(y_1-b_1)+(b_2-x_2)(y_2-b_2)=0\}.
\end{align*}
We use the hyperplane separation theorem again and compute the intersection points of $d_1(x),d_{2}(x)$ with $\Pi_j$ obtaining
\begin{align*}
\frac{1}{d_{1}(x)^{2}}+ \frac{1}{d_{2}(x)^{2}}&\geq \frac{(b_1-x_1)^2}{((b_1-x_1)^2+(b_2-x_2)^2)^2}+\frac{(b_2-x_2)^2}{((b_1-x_1)^2+(b_2-x_2)^2)^2}\\
&= \frac{1}{(b_1-x_1)^2+(b_2-x_2)^2}\geq \frac{1}{\norm{(-b)\boxplus x}^2_{\mathbb{H}}}.
\end{align*}
At that point we use that $b$ was chosen, such that $\delta_{K}(x)=\norm{(-b)\boxplus x}_{\mathbb{H}}$ is fulfilled. Summarizing our estimates we arrive at 
\begin{align*}
\left( \frac{48\sqrt{6}}{c_m}+  1  \right)^{-4/3}\int_{\Omega} \frac{|u(x)|^2}{\delta_K(x)^2} \df x \ \leq 5\int_{\Omega} |\nabla_{\mathbb{H}}\, u(x)|^{2} \df x,
\end{align*}
where Lemma \ref{lemmasharpestimate} finally yields the result.
\end{proof}
\begin{proof}[Proof of inequality {\eqref{thecmconstant}}:] Let us assume, that $c_m>0$ fulfills
\begin{align*}
  c_m^{4/3}  \pi m\sqrt{c_m^2+16} \left( 1+   \frac{c_m}{48\sqrt{6}}  \right)^{2/3}= 16/3.
\end{align*}
 It can be easily seen that
\begin{align*}
  c_m    \leq\frac{4  \sqrt[3]{2}}{\sqrt[3]{m\pi}}\leq \frac{4}{\sqrt[3]{m}}.
\end{align*}
Thus we get the following estimate
 \begin{align*}
  16/3\leq c_m \frac{4^{1/3}}{\sqrt[9]{m}}  \pi m\sqrt{\frac{16}{\sqrt[3]{m^2}}+16} \left( 1+   \frac{1}{\sqrt[3]{m}}\frac{1}{12\sqrt{6}}  \right)^{2/3}\leq  c_m  \pi m^{8/9} 4^{4/3}\sqrt{2} \left( 1+ \frac{1}{12\sqrt{6}}  \right)^{2/3},
\end{align*}
which yields the result.
\end{proof}

\section{\bf Convex polytopes with improved constants}\label{Copoimproved}
In this section we proof that for some open bounded convex polytopes the constant in Theorem \ref{maineins} can be improved. We discuss that behavior in detail for convex cylinders. In the end we show for the optimal constant $c>0$, fulfilling \eqref{ersteungleichung}, that $2\leq c$ holds, which is similar to the setting in the Euclidean case.
\subsection{The improved version}
\begin{assumption}\label{assumpconvexpol}
Let $\Omega$ be an open bounded convex polytope. Let $m\in\mathbb{N}$ denote the number of hyperplanes of $\partial \Omega$, which are not orthogonal to the hyperplane $x_3=0$. We denote these hyperplanes by $\Pi_j$ for $1\leq j \leq m$. Thus there exist  $n_{1,j},n_{2,j},n_{3,j},c_j\in\mathbb{R}$ such that
\begin{align*}
\Pi_j:=\{y\in \mathbb{H} | \ n_{1,j}y_1+n_{2,j}y_2+n_{3,j}y_3=c_j\},
\end{align*}
where $n_{3,j}\neq 0$ for $1\leq j \leq m$. We assume that there exists a constant $a>0$ such that for all $x\in \Omega$ holds
\begin{align}\label{extrassumption}
 (-2n_{2,j}/n_{3,j}+x_1)^2+(2n_{1,j}/n_{3,j}+x_2)^2\geq a|-c/n_3+x_3+x_1n_1/n_3+n_2x_2/n_3|.
\end{align}
\end{assumption}

\begin{thm}\label{assumpconvexpol1}
Under Assumption \ref{assumpconvexpol} it holds
 \begin{align}
\left( \frac{48\sqrt{6}}{a}+  1  \right)^{-4/3}\int_{\Omega}  \frac{|u(x)|^2}{\delta_C(x)^2} \df x \ \leq 4\int_{\Omega} |\nabla_{\mathbb{H}}\, u(x)|^{2} \df x
\end{align}
for all $u\in C_0^\infty(\Omega)$.
\end{thm}
\begin{proof} 
We use Lemma \ref{pseudodist} to obtain
\begin{align}
\int_{\Omega}\left( \frac{1}{d_{1}(x)^{2}}+ \frac{1}{d_{2}(x)^{2}}   \right)|u(x)|^2 \df x \ \leq 4\int_{\Omega} |\nabla_{\mathbb{H}}\, u(x)|^{2} \df x
\end{align}
for $u\in C_0^\infty(\Omega)$ and proceed in the same way as in the proof of Theorem \ref{maineins}. We treat only the case $b\in \Pi_j$ with $\delta_K(x)=\norm{(-b)\boxplus x}_{\mathbb{H}}$, since the other one is verbatim the same. By $n_j$ we denote the unit normal to $\Pi_j$. Again we use the hyperplane separation theorem and get
\begin{align*}
\frac{1}{d_{1}(x)^{2}}+ \frac{1}{d_{2}(x)^{2}} \geq \frac{1}{4}\frac{(-2n_{2,j}+x_1n_{3,j})^2+(2n_{1,j}+x_2n_{3,j})^2}{|-c_j+\sprod{x}{n_j}|^{2}}.
\end{align*}
Under Assumption \ref{assumpconvexpol} we use Proposition \ref{propohyperplane}, yielding
\begin{align*}
  \frac{1}{4}\frac{(-2n_{2,j}+x_1n_{3,j})^2+(2n_{1,j}+x_2n_{3,j})^2}{|-c_j+\sprod{x}{n_j}|^{2}} \geq    \left( \frac{48\sqrt{6}}{a}+  1  \right)^{-4/3}\left( \inf\limits_{y\in \Pi_j} \norm{(-y)\boxplus x}_{\mathbb{H}}\right)^{-2}.
\end{align*}
Since $b\in \Pi_j$ we use Lemma \ref{lemmasharpestimate} and finish the proof. 
\end{proof}
\begin{remark}The last result can be extended to any convex bounded $\Omega$ as long as there exists a constant $a>0$ such that for any hyperplane, which separates $\Omega$ from points lying on its boundary, holds \eqref{extrassumption}.
\end{remark}

\subsection{Convex cylinders}

 We discuss briefly that there are domains, which fulfill Assumption \ref{assumpconvexpol}. Therefore we consider domains of the form $\Omega=\omega\times(\alpha,\beta)$, where $\omega\subset \mathbb{R}^2$ is a bounded convex domain and $\alpha<\beta$. This domain is not a polytope, but we see that the hyperplanes, which separate the points lying in $b\in \partial\omega \times (\alpha,\beta)$ are orthogonal to the hyperplane $x_3=0$. Thus the proof of Theorem \ref{assumpconvexpol1} goes through and we get;
\begin{corollary}
 Let $\Omega=\omega\times(\alpha,\beta)$ such that $\alpha<\beta$ and  $\omega\subset \mathbb{R}^2$ is a bounded convex domain. For fixed $a>0$ we assume that for all $x\in \Omega$ holds
 $$ x_1^2 +x_2^2\geq a|-\alpha+x_3|, \quad \text{and}  \quad    x_1^2 +x_2^2\geq a|-\beta+x_3|. $$
 Then holds for all $u\in C_0^\infty(\Omega)$
 \begin{align}\label{corinequ}
\left( \frac{48\sqrt{6}}{a}+  1  \right)^{-4/3}\int_{\Omega}  \frac{|u(x)|^2}{\delta_C(x)^2} \df x \ \leq 4\int_{\Omega} |\nabla_{\mathbb{H}}\, u(x)|^{2} \df x.
\end{align}
\end{corollary}
\begin{proof}[Proof of Theorem {\normalfont \ref{1approxdom}}:] Let $a>0$ be fixed. We consider the following domain $\Omega_a:= B_{1}(p_a)\times(0,1)$, where $B_{1}(p_a)$  is the two-dimensional Euclidean ball with radius one centered at $p_a:=(\sqrt{a}+1,0)$. The conditions of the last Corollary can be checked easily, where $\alpha=0$ and $\beta=1$. 
Thus the Hardy inequality \eqref{corinequ} holds, where the constant depends on $a>0$.
\end{proof}

\subsection{On the sharp constant}

\begin{thm} \label{lem-sharp} 
Let $\Omega \subset \mathbb{H}$ be a bounded domain.  Then holds
\begin{align} 
\inf\limits_{u\in  C_{0}^{\infty}(\Omega)} \frac{\int_{\Omega} |\nabla_{\mathbb{H}}\, u(x)|^{2} \df x}{\int_{\Omega}  {|u(x)|^{2}}{{\delta_C(x)^{-2} \df x}}} \leq \frac 14\, .
\end{align}
\end{thm}

\smallskip

\begin{proof} It suffices to construct a sequence $u_n\in  C_{0}^{\infty}(\Omega)$ such that
 \begin{align} \label{enough-1}
\lim\limits_{n\to\infty} \frac{\int_{\Omega} |\nabla_{\mathbb{H}}\, u_n(x)|^{2} \df x}{\int_{\Omega}  {|u_n(x)|^{2}}{{\delta_C(x)^{-2} \df x }}} =\frac{1}{4}.
\end{align}
To this end we consider the sequence 
$$
\tilde{u}_n(x)=\delta_C(x)^{1/2+1/n}, \qquad n\in \mathbb{N}.
$$
and recall that $\delta_C(x)$ satisfies the Eikonal equation 
\begin{align} \label{eikonal}
|\nabla_{\mathbb{H}}\, \delta_C(x)|^2=1 , \quad \text{for a.e.} \ x \in {\Omega},
\end{align}
see \cite[Thm 3.1]{montiii}. Moreover, from  \cite{gromov} we know that 
\begin{align}\label{heisenbergeuclideanballs}
M \norm{x-y}_e \leq d_C(x,y) \leq M^{-1} \norm{x-y}_e^{1/2}\, .
\end{align}
holds for some $M>0$ and all $x,y \in\overline{\Omega}$. Hence the integral $\int_{\Omega} \delta_C(x)^{2/n-1} \df x<\infty$, and using \eqref{eikonal} we easily find that 
\begin{align} \label{enough-2}
\frac{\int_{\Omega} |\nabla_{\mathbb{H}}\, \tilde u_n(x)|^{2} \df x}{\int_{\Omega}  {|\tilde u_n(x)|^{2}}{{\delta_C(x)^{-2} \df x}}} = \left(\frac 12 +\frac 1n\right)^2\, \qquad \forall \ n\in\N.
\end{align}
Next we will show that $\delta_C$  is weakly differentiable with respect to $X_1$ and $X_2$ on $\Omega$. Without loss of generality we consider only the case $X_1$. Let $u\in C_{0}^\infty(\Omega)$ be given. We must show
\begin{align} \label{must-show}
\int_{\Omega}X_{1}u(x)\, \delta_C(x) \df x = -\int_{\Omega}u(x)X_1\, \delta_C(x) \df x.
\end{align}
Since we can extend these functions to the whole space we can integrate over $\mathbb{R}^3$. An application of the dominated convergence theorem yields then
\begin{align}
\int_{\mathbb{R}^3}X_{1}u(x)\delta_C(x) \df x \ = \ \lim_{h \to 0}\left(\int_{\mathbb{R}^3}\frac{u(x+h\tilde{x})}{h}\delta_C(x) \df x-\int_{\mathbb{R}^3}\frac{u(x)}{h}\delta_C(x) \df x \right),
\end{align}
where $\tilde{x}:=(1,0,x_2/2)$. We make the change of variables  $x+h\tilde{x} \mapsto x$ to obtain 
\begin{align*}
\int_{\mathbb{R}^3}X_{1}u(x)\delta_C(x) \df x \ &= \ \lim_{h \to 0}\left( \int_{\mathbb{R}^3}{u(x)}\frac{\delta_C(x-h\tilde{x})}{h} \df y-\int_{\mathbb{R}^3}\frac{u(x)}{h}\delta_C(x) \df y \right)\\
&= \ -\lim_{h \to 0}\left( \int_{\mathbb{R}^3}{u(x)}\frac{\delta_C(x-h\tilde{x})-\delta_C(x)}{-h} \df y\right).\\
\end{align*}
Since  any two points lying in $\mathbb{H}$ can be connected by a (not necessarily unique) geodesic, see \cite{monti}, we can easily deduce
\begin{align}\label{asdasdasdasfdfd}
 |\delta_C(x)-\delta_C(y)| \leq d_C(x,y) , \quad \text{for all} \ x,y \in \mathbb{H}.
\end{align}
Taking that inequality and the application of the left-invariance of the C-C distance we get $d_C(x-h\tilde{x},x)=d_C(-he_1,0)=|h|$, where $e_1:=(1,0,0) \in \mathbb{H}$. Hence we may apply the dominated convergence theorem again arriving at
\begin{align*}
\int_{\mathbb{R}^3}X_{1}u(x)\, \delta_C(x) \df x \ = \ - \int_{\mathbb{R}^3}{u(x)}\lim_{h \to 0}\frac{\delta_C(x-h\tilde{x})-\delta_C(x)}{-h} \df y.
\end{align*}
This limit exists almost everywhere on $\mathbb{H}$, see \cite{montiii}, since $\delta_C(x)$ fulfills \eqref{asdasdasdasfdfd}. This proves \eqref{must-show}, and therefore it follows that $\delta_{C}$ is weakly differentiable on $\Omega$ with respect to $X_1$. The case of $X_2$ is treated in the same way. 

At this point it can be shown by a standard argument that $\delta_C$ can be approximated by $C_0^{\infty}(\Omega)$ functions and that the same is true for the sequence $\tilde u_n$.
\end{proof}

  
\section*{\bf Aknowledgements}
B.~R. was supported by the German Science Foundation through the Research Training Group 1838: \textit{Spectral Theory and Dynamics of Quantum Systems.} The authors would like to thank Simon Larson for many fruitful discussions.



\end{document}